\documentclass[12pt]{amsart}
\usepackage{latexsym,amsxtra,amscd,ifthen}
\usepackage{amsfonts,cleveref}
\usepackage{verbatim}
\usepackage{amsmath}
\usepackage{amsthm}
\usepackage{amssymb}
\usepackage{adjustbox}
\usepackage{url}
\usepackage[arrow,matrix]{xy}
\usepackage{blkarray, bigstrut}
\usepackage{xparse}

\theoremstyle{plain}

\newtheorem{theorem}{Theorem}
\newtheorem{lemma}[theorem]{Lemma}
\newtheorem{proposition}[theorem]{Proposition}
\newtheorem{corollary}[theorem]{Corollary}

\newtheorem*{theorem*}{Theorem}
\newtheorem*{lemma*}{Lemma}
\newtheorem*{proposition*}{Proposition}
\newtheorem*{corollary*}{Corollary}
\newtheorem*{conjecture*}{Conjecture}
\newtheorem*{hypothesis*}{Hypothesis}
\newtheorem*{notation*}{Notation}

\numberwithin{theorem}{section}
\numberwithin{equation}{theorem}

\theoremstyle{definition}
\newtheorem{definition}[theorem]{Definition}

\newtheorem{example}[theorem]{Example}
\newtheorem{remark}[theorem]{Remark}

\newtheorem*{question*}{Question}

\newcommand{\C}{\mathbb{C}}
\newcommand{\N}{\mathbb{N}}
\newcommand{\Z}{\mathbb{Z}}

\DeclareMathOperator{\ch}{char}

\DeclareMathOperator{\Hom}{Hom}

\DeclareMathOperator{\Repr}{Repr}

\DeclareMathOperator{\FPdim}{FPdim}

\DeclareMathOperator{\diag}{diag}

\begin{document}
\title{\textbf{The applications of probability groups on Hopf algebras}}
\author{Jingheng Zhou, Shenglin Zhu}
\address{Jingheng Zhou: Shanghai Center for Mathematical Sciences, Fudan
University, Shanghai 200433, Shanghai, China}
\email{jhzhou14@fudan.edu.cn}
\address{Shenglin Zhu: School of Mathematical Sciences, Fudan University,
Shanghai 200433, Shanghai, China}
\email{mazhusl@fudan.edu.cn}

\begin{abstract}
In this work, we use probability groups, introduced by Harrison in 1979, as
a tool to study a semisimple Hopf algebra $H$ with a commutative character
ring and prove that the algebra generalized by the dual probability group is
the center $Z\left( H\right) $ of $H$ and the product of two
class sums is an integral combination up to a factor of $\dim (H)^{-1}$ of
the class sums of $H$. We classify all the 2-integral probability groups
with 2 or 3 elements.
\end{abstract}

\maketitle

\section{Introduction}

In the character theory of finite groups, orthogonality relations are useful
to decide whether a character is irreducible, to find the decomposition of 
a representation and so on.  R. G. Larson \cite{La1971}
generalized the first orthogonality relation to Hopf algebra in 1971 .
In 2010, M. Cohen and S. Westreich \cite{CW2010a,CW2010b}
introduced the conjagacy classes and class sums for Hopf 
algebras, and then they \cite{CW2014} generalized 
the character table to semisimple Hopf algebras in 2014. 
An interesting result about class sums is that the product of 
two class sums is no longer an integral
combination of class sums anymore, which differs from the group case.

A notion of probability group, which is introduced and developed by D. K.
Harrison \cite{Ha1979,Ha1984} since 1979, is closely related to the
representations of finite dimensional Hopf algebras. Harrison proved that all
non-isomorphic irreducible characters of one semisimple Hopf algebra $H$ 
form a probability group $A(H)$ (see \cite[Theorem 2.1]{Ha1984}). 
The relations among probability groups, probability
subgroups and dual probability groups are also deduced. In this paper, we
use these notions to study the characters of Hopf algebras. 
We recall the concept of probability groups in Section 2.

As all the finite dimensional left $H$-modules over a semisimple Hopf algebra 
$H$ form a fusion category, it is natural to generalize Harrison's
result to fusion categories. In Section 3, 
we prove that the set of all isomorphism classes of simple objects  
of a fusion category is a probability group equipped with 
a probability map, which is defined by the Frobenius-Perron 
dimension and fusion rules. Such a probability group is quasi-2-integral.

For a finite group $G$, we notice the fact that $\hat{A}(\Bbbk G)$, the
dual of $A(\Bbbk G)$, is exactly the set of class 
sums of $G$. In section 4, using the notions of conjugacy classes 
and class sums, for a semisimple Hopf 
algebra $H$ with a commutative character ring, 
we get the structure of $\hat{A}(H)$ and 
explicit expressions for the first and the second 
orthogonality relations over $H$.

In Section 5, we apply our results in Section 4 
to quasitriangular Hopf algebras. 
For a finite dimensional semisimple quasitriangular Hopf algebra $H$
over a field of characteristic zero, the Drinfeld double $D(H)$ 
is also semisimple and $A(H)$ is a probability subgroup of $A(D(H))$. 
M. Cohen and S. Westreich \cite[Theorem 2.6]{CW2011} proved 
that the product of two class sums of $H$
is an integral combination up to a factor of $\dim(H)^{-2}$ of 
the class sums of $H$, i.e.
$$C_iC_j=\frac{1}{\dim(H)^{2}}\sum_{k=1}^m \hat{N}_{ij}^k C_k,$$
where $\{C_i\mid 1\leq i\leq m\}$ is the set of class sums of $H$ and each 
$\hat{N}_{ij}^k$ is a non-negative integer.
In this section, we first prove that $A(D(H))\cong \hat{A}(D(H))$, 
and then prove that the factor $\dim(H)^{-2}$ can be replaced by 
$\dim(H)^{-1}$.

The last section is devoted to the structure of 2-integral probability groups
with 2 or 3 elements. We prove that any 2-integral probability group 
with two elements must be isomorphic to $A(\Bbbk \mathbb{Z}_2)$. And 
$A(\Bbbk \mathbb{Z}_3)$, $A(\Bbbk \mathcal{S}_3)$ are the only two types
of 2-integral probability groups with three elements.

Throughout this paper, $\Bbbk $ denotes an algebraically closed field of
characteristic zero. For an algebra $A$, $\Repr(A)$ is the finite 
dimensional left-module category, and $Z(A)$ is the center of $A$.

\section{Preliminaries}

\label{sec2}

\subsection{Hopf algebras and fusion categories}

A coalgebra $(C,\Delta ,\varepsilon )$ over $\Bbbk$ is a 
$\Bbbk$-vector space $C$ with a comultiplication 
$\Delta :C\rightarrow C\otimes C$ and a counit 
$\varepsilon :C\rightarrow \Bbbk $ satisfying 
the following commuting diagrams: 
\[
\xymatrix{ C\ar[d]_{ \Delta} \ar[rr]^{\Delta} & &C\otimes C \ar[d]^{\Delta
\otimes I} \\ 
C\otimes C \ar[rr]^{I\otimes \Delta} & &C\otimes C\otimes C }
\quad\quad
\xymatrix{ \Bbbk\otimes C & C\otimes C 
\ar[l]_{\varepsilon\otimes I}\ar[r]^{I\otimes\varepsilon} & C\otimes \Bbbk\\
&C\ar[ul]\ar[u]^{\Delta}\ar[ur]& }
\]

A bialgebra $(B,M,u,\Delta ,\varepsilon )$ is a vector space $B$ with both an
algebra structure $(B,M,u)$ and a coalgebra structure $(B,\Delta
,\varepsilon )$, such that $\Delta $ and $\varepsilon $ are algebra maps.
A Hopf algebra $H$ is a bialgebra with an antipode $S:H\rightarrow H$, which
is the convolution inverse of the identity map in 
$\mathrm{Hom}_{\Bbbk}\left( H,H\right) $.

For a finite dimensional Hopf algebra $H$ over a field of characteristic
zero, by theorems of R. G. Larson and D. E. Radford 
\cite{LR1988a,LR1988b}, the following statements are equivalent:

\begin{enumerate}
\item[(1)] $H$ is semisimple,

\item[(2)] $H$ is cosemisimple,

\item[(3)] $S^{2}=1$,

\item[(4)] A non-zero left integral $t$ of $H$ is cocommutative. 
\end{enumerate}

A quasitriangular Hopf algebra $(H,R)$ is Hopf algebras $H$ with an
invertible R-matrix $R\in H\otimes H$, satisfying

\begin{enumerate}
\item[(1)] $R\Delta(h)=\Delta^{op}(h)R$ for all $h\in H$,

\item[(2)] $(\Delta\otimes id) R=R^{13}R^{23}$,

\item[(3)] $(id\otimes\Delta) R=R^{13}R^{12}$.
\end{enumerate}

For a quasitriangular Hopf algebra $(H,R)$, by a result of V. G. Drinfeld 
\cite{Dr1990}, there exists a surjective map $\Phi :D(H)\longrightarrow H$,
such that for 
$p\bowtie h\in D(H), \Phi(p\bowtie h)=\langle p,R^{1}\rangle R^{2}h.$

Furthermore, the Drinfeld double $D(H)$ of any finite dimensional Hopf
algebra $H$ is quasitriangular with 
\begin{equation*}
R_{D(H)}=\displaystyle\sum_{i=1}^{d}(\varepsilon \bowtie h_{i}^{\ast
})\otimes (h_{i}\bowtie 1_{H}),
\end{equation*}
where $d=\dim (H)$, and $\{h_{i}\}_{i=1}^{d}$ and 
$\{h_{i}^{\ast}\}_{i=1}^{d}$ are dual bases of $H$ and $H^{\ast }$ respectively.
By \cite[Theorem 2.10]{RS1988}, $D(H)$ is factorizable with 
the isomorphism of vector spaces $F:(D(H))^{\ast }\longrightarrow D(H)$
with 
\begin{equation*}
F(p)=\sum\langle p,R_{D(H)}^{1}r_{D(H)}^{2}\rangle R_{D(H)}^{2}r_{D(H)}^{1},
\quad \forall p\in D(H)^*, 
\end{equation*}
where $r_{D(H)}=R_{D(H)}$.

A monoidal category is a quintuple $(\mathcal{C},\otimes, a, {\bf 1}, \iota)$, 
where $\mathcal{C}$ is a category, 
$\otimes :\mathcal{C}\times \mathcal{C} \rightarrow \mathcal{C}$ 
is a bifunctor called the tensor product bifunctor, $a:(\cdot \otimes
\cdot )\otimes \cdot \rightarrow \cdot \otimes (\cdot \otimes \cdot )$ 
is a natural isomorphism called the associativity constraint, 
${\bf 1}\in \mathcal{C}$ is an object of $\mathcal{C}$, and 
$\iota: {\bf 1} \otimes {\bf 1}\to {\bf 1}$ is an isomorphism, subject to 
the pentagon axiom and the unit axiom.

A tensor category $\mathcal{C}$ over $\Bbbk$ is a locally finite 
$\Bbbk$-linear abelian rigid monoidal category, in which the bifunctor $\otimes$
is bilinear on morphisms and $End(\mathbf{1})\cong \Bbbk$. 
An example for tensor category is the
representation category of a Hopf algebra. A fusion category is a semisimple
tensor category with finitely many isomorphism classes of simple objects. 
Hence, the representation category of a finite dimensional semisimple 
Hopf algebra is a fusion category.

\subsection{Probability groups}

We first recall some definitions and properties of probability groups
introduced and developed by Harrison \cite{Ha1979,Ha1984}.

\begin{definition}[{\cite[pp150]{Ha1984}}]
\label{def2.1} A probability group means a set $A$ equipped with a map 
$p:(a,b,c)\mapsto p(a\cdot b=c)$ from $A\times A\times A$ to non-negative
reals $\mathbb{R}_{\geq 0}$ such that:

\begin{enumerate}
\item[(1)] for all $a, b\in A$, $p(a\cdot b=c)=0$ for all but finitely many 
$c$ in $A$, and 
\begin{equation*}  \label{E2.1.1}
\displaystyle\sum_{c\in A} p(a\cdot b=c)=1,  \tag{E2.1.1}
\end{equation*}

\item[(2)] for all $a, b, c, d\in A$, 
\begin{equation*}  \label{E2.1.2}
\sum_{x\in A} p(a\cdot b=x)p(x\cdot c=d)=\sum_{y \in A} p(a\cdot y=d)p(b\cdot c=y),
\tag{E2.1.2}
\end{equation*}

\item[(3)]there is an element $1\in A$ such that for any $ a\in A$, 
\begin{equation*}  \label{E2.1.3}
p(1\cdot a=a)=1=p(a\cdot 1=a),  \tag{E2.1.3}
\end{equation*}

\item[(4)]for any $a\in A$, there exists a unique element $a^{-1}\in A$
such that 
\begin{equation*}  \label{E2.1.4}
p(a\cdot a^{-1}=1)> 0,  \tag{E2.1.4}
\end{equation*}

\item[(5)] for all $ a, b, c\in A$, 
\begin{equation*}  \label{E2.1.5}
p(a\cdot b=c)=p(b^{-1}\cdot a^{-1}=c^{-1}), \tag{E2.1.5}
\end{equation*}

\item[(6)] for all $a\in A$, 
\begin{equation*}  \label{E2.1.6}
p(a\cdot a^{-1}=1)=p(a^{-1}\cdot a=1). \tag{E2.1.6}
\end{equation*}
\end{enumerate}
The map $p$ is called a {\it probability map} over $A$.
\end{definition}

Let $A$ be a probability group. For $a\in A$, the \textit{size} of $a$, 
denoted by $s(a)$, is the reciprocal
of $p(a\cdot a^{-1}=1)$. When $A$ is finite, \eqref{E2.1.6} follows from 
\eqref{E2.1.1}-\eqref{E2.1.5}, which was proved in \cite{Ha1979}, and the 
\textit{order} of $A$, denoted by $n(A)$, is defined as
$\displaystyle\sum_{a\in A} s(a)$. 
$A$ is called \textit{abelian} if $p(a\cdot b=c)=p(b\cdot a=c)$ for 
all $a,b,c\in A$.

Moreover, associated with any  probability
group $A$, a $\mathbb{C}$-algebra $\mathbb{C}(A)$ can be defined as a  
$\mathbb{C}$-vector space with the basis $\{a\mid a\in A\}$, in which the
multiplication is
\begin{equation*}
\label{E2.1.7}\tag{E2.1.7}
a\cdot_{p}b=\sum_{c\in A}p(a\cdot b=c)c\quad \forall a,b\in A.
\end{equation*}

\begin{example}
\label{ex2.2} There are two classical examples:
\begin{enumerate}
\item[(1)] Groups are probability groups. Indeed, given a group $G$, 
for any $g,h,k\in G$, let $p(g\cdot h=k)=\delta_{gh,k}$. It's clear that $G$
is a probability group.

\item[(2)] (\cite[Theorem 2.1]{Ha1984}) Let $H$ be a finite dimensional 
semisimple Hopf algebra over $\Bbbk$, then the set of all
isomorphism classes of irreducible $H$-modules 
is a probability group, which is denoted by $A(H)$.
\end{enumerate}
\end{example}

\begin{definition}[{\protect\cite[pp466]{Ha1979}}]
\label{def2.3} Let $A$ be a probability group. A subset $S\subset A$ is
called a \textit{probability subgroup} of $A$ if

\begin{enumerate}
\item[(1)] $1\in S$,

\item[(2)] for all $ s\in S$, $s^{-1}\in S$, and

\item[(3)] for any $ a,b\in S$ and $ c\in A$, if $p(a\cdot b=c)>0$, then 
$c\in S$.
\end{enumerate}
\end{definition}

\begin{example}
\label{ex2.4} In Example \ref{ex2.2} (1), let $S\leq G$, then $S$ is a
probability subgroup of $G$.
\end{example}

Let $S$ be a probability subgroup of $A$. For $a,b\in A$, write $a\sim b$
if there exist $s_1,s_2\in S$ and $x\in A$, 
such that $p(a\cdot s_1=x)> 0$ and 
$p(x\cdot b^{-1}=s_2)> 0$. It's not too hard to check 
'$\sim$' is an equivalence relation.

Let $[a]=\{b\in A\mid b\sim a\}$ and $A//S=\{[a]\mid
a\in A\}$. Then $[1]=S$, and by \cite[Proposition 2.2 and 2.3]{Ha1979},
we have

\begin{proposition}
\label{pro2.5} Let $A$ be a finite abelian probability group and 
$S$ be a probability subgroup of $A$. For 
$X,Y,Z\in A//S$, let $P(X\cdot Y=Z)=\displaystyle\sum_{c\in Z}p(a\cdot b=c)$, 
where $a\in X$ and $b\in Y$. Then $P(X\cdot Y=Z)$ is independent 
of the choice of $a\in X$ and $b\in Y$, and $(A//S, P)$ is a probability group.
Furthermore, $n(S)n(A//S)=n(A)$.
\end{proposition}

\begin{example}
\label{ex2.6} Let $H=\Bbbk \mathcal{S}_{3}$. The irreducible
characters of $\mathcal{S}_{3}$ are listed as follows. 
\begin{equation*}
\begin{array}{cccc}
\hline
\mathcal{S}_{3} & id & (12) & (123) \\ \hline
\chi _{1} & 1 & 1 & 1 \\ 
\chi _{2} & 1 & -1 & 1 \\ 
\chi _{3} & 2 & 0 & -1 \\ \hline
\end{array}
\end{equation*}

Thus $A(H)=\{\chi_1,\chi_2,\chi_3\}$ is a probability group with the 
probability map
\begin{equation*}
p(\chi_i\cdot \chi_j=\chi_k)=\frac{N_{ij}^k \chi_k(1)}{\chi_i(1) \chi_{j}(1)},
\quad 1\leq i,j,k\leq 3,
\end{equation*}
where $\chi_i \chi_j=\sum\limits_{k=1}^3 N_{ij}^k \chi_k$.
Moreover, $S=\{\chi_1,\chi_2\}$ is a probability subgroup of $A(H)$ and
\begin{equation*}
A(H)//S=\{[\chi_1],[\chi_3]\},
\end{equation*}
with the unit $[\chi_1]$ and 
$P([\chi_3]\cdot[\chi_3]=[\chi_1])=\displaystyle\frac{1}{2}$, 
$P([\chi_3]\cdot[\chi_3]=[\chi_3])=\displaystyle\frac{1}{2}$.
\end{example}

Next are some definitions related to the dual of a probability group $A$.

\begin{definition}[{\cite[pp476]{Ha1979}}]
\label{def2.7} A \textit{functional} of $A$ is a map $f:A\rightarrow 
\mathbb{C}$ with $f(1)=1$ and satisfying 
\begin{equation*}
f(a)f(b)=\sum_{c\in A}p(a\cdot b=c)f(c).
\end{equation*}
The \textit{dual} of $A$ is the set $\hat{A}=\{f \mid f$ is a functional of $A\}.$
\end{definition}

Any functional $f$ can be linearly extended to an algebra map 
from $\mathbb{C}(A)$ to $\mathbb{C}$. Hence $\hat{A}$ 
corresponds to the set of all algebra maps
from $\mathbb{C}(A)$ to $\mathbb{C}$. If $A$ is finite, we have $\left\vert 
\hat{A}\right\vert \leq \left\vert A\right\vert $  since $\mathbb{C}(A)$ 
is a semisimple $\mathbb{C}$-algebra \cite[pp476]{Ha1979}. 

Assume that $A$ is abelian, then the algebra $Map(A,\mathbb{C})$ 
of maps from $A$ to $\C$ with  component-wise 
addition and multiplication, i.e.
\begin{equation}
\begin{split}
(f+g)(a)& =f(a)+g(a),\\
(f\cdot g)(a)& =f(a)g(a),
\end{split}\quad \forall f,g \in Map(A,\C), a\in A,
\tag{E2.7.1}  \label{E2.7.1}
\end{equation}
and $\hat{A}$ is a basis of $Map(A,\C)$. 

So there exist uniquely defined
complex numbers $\hat{p_{\theta }}(\chi ,\psi )$, 
$\theta,\chi,\psi\in \hat{A}$, with 
\begin{equation*}
\chi \cdot \psi =\sum_{\theta \in \hat{A}}\hat{p}_{\theta }(\chi ,\psi
)\theta , \quad \forall \chi ,\psi ,\in \hat{A}.
\end{equation*}

$A$ is called \textit{dualizable} if $\hat{p_{\theta}}(\chi,\psi)$ is a
non-negative real number for all $\theta,\chi,\psi\in \hat{A}$ or $\hat{p}$
is a probability map over $\hat{A}$. In this case, $\hat{A}$ is called
the {\it dual probability group} of $A$.

If $S$ is a probability subgroup of $A$, then $S^{\bot}=\{\chi\in \hat{A}\mid
\chi(s)=1, \forall s\in S\}$.

Define $aug: A\mapsto \C$ by $aug(a)=1$ for all $a\in A$. Then an important result is

\begin{proposition}[{\protect\cite[Proposition 2.10]{Ha1979}}]
\label{pro2.8} Let $A$ be a finite abelian probability group. Define 
$\chi^{-1}(a)=\chi(a^{-1})$ for $\chi\in\hat A$. Then $\chi^{-1},aug\in \hat
A $.

Moreover, let $\hat{s}(\chi)$ be the reciprocal of 
$\hat{p}_{aug}(\chi,\chi^{-1})$, then 
\begin{eqnarray*}
\sum_{a\in A} s(a)\hat{s}(\chi)\chi(a)\psi^{-1}(a) &=& n(A)\delta_{\chi,\psi}, 
\quad \forall \chi,\psi\in \hat A,\\
\sum_{\chi\in \hat{A}} \hat{s}(\chi)s(b)\chi(a)\chi^{-1}(b) &=& n(A)\delta_{a,b},
\quad \forall a,b\in A.
\end{eqnarray*}
\end{proposition}

\section{Probability group over fusion categories}

\label{sec3}

In this section, we will construct a probability group from a fusion
category, which covers the cases in Example \ref{ex2.2}.

Given a fusion category $\mathcal{C}$, the Grothendieck ring
 $Gr(\mathcal{C})$ is the ring generated by the isomorphism classes 
 $\{[X_{i}] \mid i=1, \cdots, m\}$
of simple objects of $\mathcal{C}$, where the multiplication is defined as 
\begin{equation*}
[X_{i}][X_{j}]=\sum_{k=1}^m [X_{i}\otimes X_{j}:X_{k}][X_{k}],
\end{equation*}
in which the non-negative integer $[X_{i}\otimes X_{j}:X_{k}]$ is the
multiplicity of $X_{k}$ in $X_{i}\otimes X_{j}$. When no confusion 
is possible, we write $X$ instead of $[X]$, $m X$ ($m\geq 0$) 
instead of the direct sum of $m$ copies of $X$ and $m \Hom(X,Y)$ instead
of the direct sum of $m$ copies of $\Hom(X,Y)$ for all $X,Y\in \mathcal{C}$.

For simplify, we denote $N_{ij}^k=[X_i\otimes X_j:X_k].$

Recall that the Frobenius-Perron dimension of $X_i$
(\cite[pp618]{ENO2005}) is defined to be the largest positive eigenvalue of the
left multiplication matrix of $[X_i]=(N_{ij}^k)$ in $Gr(\mathcal{C})$, 
denoted by $\FPdim(X_i)$. 

\begin{theorem}
\label{thm3.1} Let $\mathcal{C}$ be a fusion category. Let 
$A(\mathcal{C})=\{X_1,X_2,\cdots, X_m\}$ 
be the set of all simple objects up to
isomorphism. Then $A(\mathcal{C})$ is a probability group with 
\begin{equation}  
\label{E3.1.1}\tag{E3.1.1}
p(X_i\cdot X_j= X_k)=\frac{N_{ij}^k\FPdim(X_k)}{\FPdim(X_i)\FPdim(X_j)},
\quad 1\leq i,j,k\leq m.
\end{equation}
\end{theorem}

\begin{proof}
We will check the conditions in Definition \ref{def2.1} one by one.

(1) By \cite[Theorem 8.6]{ENO2005}, $\FPdim$ is an algebra morphism 
from $Gr(\mathcal{C})$ to $\mathbb{R}$, and then we have 
$$
\FPdim(X_i) \FPdim(X_j)= \sum_{k=1}^m N_{ij}^k \FPdim(X_k).
$$ 
Therefore, \eqref{E2.1.1} holds, i.e. 
$\sum\limits_{k=1}^m p(X_i\cdot X_j= X_k)=1$.

(2) \eqref{E2.1.2} holds since the associativity constraint is a natural isomorphism.

(3) Since the unit object \textbf{1} is also a simple object with $\FPdim(\textbf{1})=1$
in $\mathcal{C}$,
then for $1\leq i\leq m$, 
\begin{equation*}
p(X_{i}\otimes \mathbf{1}=X_i)=1=p(\mathbf{1}\otimes X_{i}=X_i).
\end{equation*}
So \textbf{1} is the unit element in $A(\mathcal{C})$ and, without loss of
generality, let $X_{1}=\mathbf{1}$.

(4) By \cite[Proposition 2.10.8 ]{EGNO2015}, since a fusion category is a rigid
monoidal category, for any $X,Y,Z\in \mathcal{C}$, 
\begin{equation*}
\Hom(X\otimes Y,Z)\cong \Hom(X,Z\otimes Y^{\ast }).
\end{equation*}
Then for any $1\leq i,j\leq n$, 
\begin{eqnarray*}
\Hom(X_{i},X_{j}^{\ast }) &\cong &\Hom(X_{i},X_{1}\otimes X_{j}^{\ast }) \\
&\cong &\Hom(X_{i}\otimes X_{j},X_{1}) \\
&\cong &\bigoplus_{k=1}^m N_{ij}^{k}\Hom(X_{k},X_{1}) \\
&=&N_{ij}^{1}\Hom(X_{1},X_{1}).
\end{eqnarray*}
Hence, $N_{ij}^{1}=\delta _{X_{i},X_{j}^{\ast }}$. This implies 
$p(X_{i}\cdot X_{j}=X_{1})>0$ if and only if $X_{j}^{\ast }=X_{i}$. Then 
the inverse of $X_{i}$ exists, which is $X_{i}^{\ast }$, denoted
by $X_{i^{\ast }}$.

(5) By \cite[Proposition 2.10.7]{EGNO2015}, 
\begin{equation*}
(X_i\otimes X_j)^*=X_j^*\otimes X_i^*.
\end{equation*}
Hence, in $Gr(\mathcal{C})$, 
\begin{equation*}
(\bigoplus_{k=1}^m N_{ij}^kX_k)^*= \bigoplus_{l=1}^m N_{j^* i^*}^l X_l,
\end{equation*}
which implies $N_{ij}^k=N_{j^*i^*}^{k^*}.$

Moreover, $N_{ij}^k=N_{i^*k}^{j}$ by $N_{i i^{\ast }}^1=1$,
and then we have $\FPdim(X_i)=\FPdim(X_{i^*})$ and \eqref{E2.1.5} holds.

(6) Since $N_{i i^{\ast }}^1=1$, we have
\begin{equation*}
p(X_{i}\cdot X_{i}^{\ast }=X_{1})=\frac{1}{\FPdim(X_{i})\FPdim(X_{i}^{\ast })}.
\end{equation*}
So \eqref{E2.1.6} holds and the size of $X_{i}$ is $\FPdim^{2}(X_{i})$.

As a conclusion, $A(\mathcal{C})$ is a probability group with order 
\begin{equation*}  \label{E3.1.2}
n(A(\mathcal{C}))=\sum_{i=1}^n \FPdim^{2}(X_i)=\FPdim(\mathcal{C}). 
\tag{E3.1.2}
\end{equation*}
\end{proof}

By \eqref{E2.1.7}, there exists an isomorphism 
\begin{equation*}
\Psi :\C (A(\mathcal{C}))\longrightarrow \C \otimes_{\Z} Gr(\mathcal{C})
\end{equation*}
which maps $X_{i}$ to $\displaystyle\frac{X_{i}}{\FPdim(X_{i})}$. 

For convenience, we rewrite $p(X_{i}\cdot X_{j}=X_{k})$ as $p_{k}(i,j)$. 

\begin{example}
\label{ex3.2} 
(1) When $H=(\Bbbk G)^{\ast }$ for a finite group $G$, $A(\Repr(H))$ 
is the probability group in Example \ref{ex2.2} (1). 

(2) When $\mathcal{C}=\Repr(H)$ for a finite dimensional
semisimple Hopf algebra $H$, $\FPdim(M)=\dim (M)$ for $M\in \Repr(H)
$, thus $A(\mathcal{C})$ is the probability group $A(H)$ in Example \ref{ex2.2} (2).
\end{example}

Note that each $\FPdim(X_i)$ is an algebraic integer, then we have

\begin{definition}
\label{def3.3} A probability group $A$ is called \textit{quasi-$r$-integral}
for a positive integer $r$ if

\begin{enumerate}
\item[(1)] for all $ a\in A$, $s(a)^{\frac{1}{r}}$ is an algebraic integer, and

\item[(2)] for all $ a,b,c\in A$, $p(a\cdot b=c)s(a)^{\frac{1}{r}}s(b)^{\frac{1}{r}}
/s(c)^{\frac{1}{r}}$ is a non-negative integer.
\end{enumerate}
\end{definition}

\begin{remark}[{\cite[pp463]{Ha1979}}]
\label{rem3.4}
A quasi-$r$-integral probability group $A$ is \textit{$r$-integral} if 
for any element $ a\in A$, $s(a)^{\frac{1}{r}}$ is an integer.
\end{remark}

Since $0\leq p(a\cdot b=c)\leq 1$, we have the following lemma.

\begin{lemma}
\label{lem3.4} If $A$ is a quasi-$r$-integral probability group, then $r=1$
or $2$.
\end{lemma}

\begin{proof}
If for any $a\in A$, $p(a\cdot a^{-1}=1)=1$, then 
\begin{equation*}
p(a\cdot b=c)=p(a\cdot b=c)s(a)^{\frac{1}{r}}s(b)^{\frac{1}{r}}
/s(c)^{\frac{1}{r}}
\end{equation*}
is an integer by Definition \ref{def3.3} (2). This implies that $p(a\cdot
b=c)=0$ or $1$. For any $a,b\in A$, since $\sum\limits_c p(a\cdot b=c)=1$, 
there exists a unique $c\in A$ such that $p(a\cdot b=c)=1$. 
Define $a\times b=c$, and then $(A,\times)$ is a group. Hence $r=1$.

Suppose there is an element $a\in A$, 
such that $p(a\cdot a^{-1}=1)< 1$. Let $b=a^{-1}$, 
$c=1$, and denote $m=s(a)^{\frac{1}{r}}$. We have $p(a\cdot a^{-1}=1)m^{2}$
is a positive integer since 
$s(a)=s(a^{-1})=\displaystyle\frac{1}{p(a\cdot a^{-1}=1)}> 1$.
Namely, $m^{2-r}$ is a positive integer. 
However, $m=s(a)^{\frac{1}{r}}> 1$, which forces that $r=1$ or $2$.
\end{proof}

Obviously, $A(\mathcal{C})$ is quasi-2-integral.

\section{Dual probability groups of Hopf algebras}

\label{sec4}

In this section, we will focus on $\hat{A}(\mathcal{C})$, the dual of
probability group $A(\mathcal{C})$. A first step is to study it when
$\mathcal{C}=\Repr(\Bbbk G)$ for a finite group $G$. 

\begin{example}
\label{ex4.1} By calculating $\hat{A}(\Bbbk \mathcal{S}_3)$ 
(see Example \ref{ex2.6}), we obtain that
$\hat{A}(\Bbbk \mathcal{S}_3)=\{f_{1},f_{2},f_{3}\}$ where 
\begin{equation*}
\begin{array}{ccccc}
\hline
& \vline & \frac{\chi _{1}}{\chi _{1}(1)} & \frac{\chi _{2}}{\chi _{2}(1)} & 
\frac{\chi _{3}}{\chi _{3}(1)} \\ 
\hline
f_{1} & \vline & 1 & 1 & 1 \\ 
f_{2} & \vline & 1 & -1 & 0 \\ 
f_{3} & \vline & 1 & 1 & -\frac{1}{2} \\ \hline
\end{array}
\end{equation*}
It is not hard to check $\hat{A}(\Bbbk \mathcal{S}_3)$ is a probability group.

On the other hand, let 
$C_{1}=\{id\},C_{2}=\{(12),(23),(13)\},C_{3}=\{(123),(132)\}$, 
the three conjugacy classes in $\mathcal{S}_{3}$. 
Take $c_{i}=\displaystyle\frac{\sum_{g\in C_{i}}g}{|C_{i}|}$, and then by regarding 
$c_{i}$ $\left( 1\leq i\leq 3\right) $ as elements in $(\Bbbk \mathcal{S}_3)^{**}$,  
$f_{i}(\frac{\chi _{i}}{\chi _{i}(1)} )=c_{i}(\frac{\chi _{i}}{\chi _{i}(1)} )$. 
Moreover, the map 
$\phi: \Bbbk \hat{A}(\Bbbk \mathcal{S}_3)\to Z(\Bbbk \mathcal{S}_3)$
with $\phi(f_i)=c_i$ is an algebra isomorphism.
\end{example}

Therefore, we have the following proposition:
\begin{proposition}
\label{pro4.2} Let $G$ be a finite group with conjugacy classes 
$\{\mathcal{C}_i\}_{i=1}^m$ and 
$c_i=\displaystyle\frac{\sum_{g\in \mathcal{C}_i}g}{|\mathcal{C}_i|}$. 
Let $\Gamma(G)=\{\chi_1,\ldots, \chi_m\}$ be the set of all non-isomorphic
irreducible characters of $G$. Then $\hat{A}(\Bbbk G)=\{c_i \mid
i=1,\ldots,m\}$ and $\hat A(\Bbbk G)$ is a probability group. Furthermore, 
$\Bbbk \hat{A}(\Bbbk G)=Z(\Bbbk G)$. 
\end{proposition}

\begin{proof}
To show $c_{l}\in \hat{A}(\Bbbk G)$, it suffices  to show the following
equality for all $1\leq i,j\leq m$,
\begin{equation*}
\sum_{k}p_k(i,j)\langle c_{l},x_{k}\rangle =\langle c_{l},x_{i}\rangle
\langle c_{l},x_{j}\rangle \text{,}
\end{equation*}
where $x_i=\displaystyle\frac{\chi_i}{\chi_i(1)}$.

In fact, for any $g\in \mathcal{C}_l$, $\chi_i,\chi_j\in \Gamma(G)$, 
\begin{equation*}
\langle \chi_i, g \rangle \langle \chi_j, g\rangle =
\langle \chi_i \chi_j, g\rangle \text{ and } \langle  \chi_i, g\rangle =
\langle  \chi_i, c_l\rangle,
\end{equation*}
then 
\begin{equation*}
\langle  \chi_i, c_l\rangle \langle \chi_j, c_l\rangle =
\langle \chi_i \chi_j, c_l\rangle = \langle \sum_k N_{ij}^k \chi_k, c_l \rangle.
\end{equation*}
Therefore $c_l\in \hat{A}(\Bbbk G)$ since 
$x_i=\displaystyle\frac{\chi_i}{\chi_i(1)}$.

On the other hand, 
\begin{equation*}
m\leq \left\vert \hat{A}(\Bbbk G)\right\vert \leq \left\vert A(\Bbbk
G)\right\vert =m,
\end{equation*}
thus $\hat{A}(\Bbbk G)=\{c_{i}\mid i=1,\ldots ,m\}$.

By the definition of $c_i$, $C_i=|\mathcal{C}_i|c_i$ is 
the class sum of the conjugacy class $\mathcal{C}_i$. 
Hence, there exists a family of non-negative integers 
$\{\hat{N}^k_{ij} \mid 1\leq i,j,k\leq m\}$ such that 
\begin{equation*}
C_i C_j=\sum_{k=1}^m \hat{N}^k_{ij} C_k.
\end{equation*}

Moreover, since $\Bbbk G$ is a semisimple algebra,  let $\{e_i\}$ be 
all the primitive central idempotents of $\Bbbk G$ with 
$\chi_j(e_i)=\delta_{ij}\chi_j(1)$,
then $\{e_i\}$ forms a basis of $Z(\Bbbk G)$, where 
$Z(\Bbbk G)$ is the center of $\Bbbk G$. Then for $ 1\leq l,i,j\leq m$, 
\begin{equation*}
x_{l}(e_i e_j)=x_{l}(e_i) x _{l}(e_j).
\end{equation*}
It means 
\begin{equation*}
x_{l}(c_{i}c_{j})=x_{l}(c_{i}) x_{l}(c_{j}).
\end{equation*}
Hence, in $\Bbbk \hat{A}(\Bbbk G)$, the product $c_i\cdot c_j$
is exactly the product $c_ic_j$ in $Z(\Bbbk G)$, and then we have
\begin{eqnarray*}
c_{i}\cdot c_{j} &=&c_{i}c_{j} \\
&=&\frac{C_{i}}{|\mathcal{C}_{i}|}\frac{C_{j}}{|\mathcal{C}_{j}|} \\
&=&\sum_{k=1}^{m}\frac{\hat{N}^k_{ij}|\mathcal{C}_{k}|}{|\mathcal{C}_{i}| |\mathcal{C}_{j}|}c_{k}.
\end{eqnarray*}
As a consequence, $\hat{A}(\Bbbk G)$ is a probability group with 
probability map
\begin{equation*}
\hat{p}_{k}(i,j)=\frac{\hat{N}^k_{ij}|\mathcal{C}_{k}|}{|\mathcal{C}_{i}| |\mathcal{C}_{j}|}
\end{equation*}
and $\Bbbk \hat{A}(\Bbbk G)=Z(\Bbbk G)$.
\end{proof}

To generalize Proposition \ref{pro4.2} to Hopf case, let us
recall the definition of conjugacy classes and class sums for Hopf
algebra, which were introduced by M. Cohen and S. Westreich
\cite{CW2010a,CW2010b}.

Let $H$ be a semisimple Hopf algebra and $\Gamma =\{\chi _{1},\cdots ,\chi
_{m}\}$, the set of irreducible characters (up to isomorphism) of $H$.
Let $C(H)$ be the subalgebra of all cocommutative elements in $H^{\ast }$,
then it's well-known that $C(H)\cong \Bbbk \otimes _{\mathbb{Z}}Gr(H)$.  
We denote by $\{E_{1},\cdots ,E_{m}\}$ the
set of orthogonal central primitive idempotents of $C(H)$.

\begin{definition}[{\protect\cite[Definition 3.3]{CW2010a}, 
\protect\cite[Definition 2.2]{CW2010b}}]
\label{def4.3} Let $t$ be the integral in $H$ with $\varepsilon (t)=1$. Then
the $\mathit{i}$\textit{-th class sum} $C_{i}$ is defined as 
\begin{equation*}
C_{i}=dE_{i}\rightharpoonup t\text{,}
\end{equation*}
where $d=\dim (H)$. Define the \textit{conjugacy class} $\mathcal{C}_{i}$ by 
\begin{equation*}
\mathcal{C}_{i}=H^{\ast }E_{i}\rightharpoonup t
\end{equation*}
which is the right coideal of $H$ generated by $C_{i}$.
\end{definition}

\begin{remark}
\label{rem4.4} Since $C(H)$ is $S^{\ast }$-stable,  if $E$ is a central
primitive idempotent in $C(H)$, then so is $S^{\ast }(E)$. Therefore, for 
$1\leq i\leq m$, let $i^{\ast }$ be the number which satisfies $E_{i^{\ast
}}=S^{\ast }(E_{i})$.
\end{remark}

Note that each $\mathcal{C}_{i}$ is a right $H$-comodule via the coaction `
$\Delta $' and a left $H$-module via the action '$\cdot _{ad}$', where 
$h\cdot _{ad}c=\sum h_{(2)}cS(h_{(1)})$. Then, by \cite[Lemma 1.1(ii)]{CW2011}
, $\mathcal{C}_{i}\in ~_{H}\mathcal{YD}^{H}$.

Assume that $C(H)$ is commutative. 
Let $d_{i}=\dim (\mathcal{C}_{i})$, then $\{E_{i},c_{i}:=\frac{C_{i}}{d_{i}}
\}_{i=1}^{m}$ are dual bases for $C(H)$ and $Z(H)$ respectively by 
\cite[pp105]{CW2010a}.

\begin{lemma}
\label{lem4.5} For any $x\in \mathcal{C}_i$, $t\cdot_{ad}
x=\varepsilon(x)c_i $.
\end{lemma}

\begin{proof}
It is directly from \cite[Proposition 3.6 (ii), (iii)]{CW2010a}.
\end{proof}

Let $K$ be the commutator subspace of $H$, 
i.e. $K=span\{ab-ba\mid a,b\in H\}$, 
then by \cite[Theorem2.7]{CW2010b}, 
\begin{equation*}
K=\bigoplus_{i=1}^m (\mathcal{C}_{i}\cap \ker \varepsilon ).
\end{equation*}

By \cite[Lemma 1.11 (ii)]{CW2010b}, $K=C(H)^{\perp }=\{h\in H\mid
f(h)=0,\forall f\in C(H)\}$. Hence, $K$ is a coideal by \cite[Proposition
1.4.6 (b)]{Sw1969}. Therefore, $C(H)^*\cong H/K$ as coalgebra and 
$H/K$ is cosemisimple since $C(H)$ is semisimple \cite{Ka1972, Zy1994}.

Note a fact of $\mathcal{C}_i$ given by M. Cohen and 
S. Westreich:

\begin{theorem}[{\protect\cite[Theorem 1.4]{CW2011}}]
\label{thm4.6} As a $D(H)$-module, $\mathcal{C}_i$ is simple and 
\begin{equation*}
H=\bigoplus_{i=1}^m \mathcal{C}_i.
\end{equation*}
\end{theorem}

Therefore, as a right $H/K$-comodule,
\begin{equation*}
H/K\cong \bigoplus_{i=1}^m\mathcal{C}_i/(\mathcal{C}_{i}\cap K)
=\bigoplus_{i=1}^m\mathcal{C}_i/ (\mathcal{C}_{i}\cap \ker \varepsilon )
=\bigoplus_{i=1}^m\Bbbk \overline c_i, 
\end{equation*}
where $\overline c_i$ is the image of $c_i$ under the natural coalgebra
morphism 
\begin{equation*}  
\pi:H\rightarrow H/K, h\mapsto \overline h.  
\end{equation*}

So each component $\Bbbk \overline c_i$ is a 1-dimensional right 
coideal of $H/K$. Hence $\bar{c_i}$ is a group-like in $H/K$ 
since $\varepsilon(c_i)=1$.

It's clear that $c_{i}\in Z(H)$, and $\{c_{i}\}_{i=1}^{m}$ forms a basis of 
$Z(H)$. Furthermore, since $\overline{c_{i}}$ is a group-like in $H/K$, 
then for any $\chi _{\alpha },\chi _{\beta }\in \Gamma $, 
\begin{equation}
\langle \chi _{\alpha }\cdot \chi _{\beta },c_{i}\rangle =\langle 
\chi_{\alpha },c_{i}\rangle \langle \chi _{\beta },c_{i}\rangle .  \tag{E4.6.1}
\label{E4.6.1}
\end{equation}

Let $e_i$ be the primitive central idempotents in $H$ corresponding to
the character $\chi_i$ for $1\leq i\leq m$, 
then $\{e_i\mid 1\leq i\leq m\}$ is a basis of $Z(H)$.
Let $E=(a_{ij})$ be the change of basis matrix from the basis $\{e_{i}\}$ 
to the basis $\{c_{i}\}$ in $Z(H)$, i.e. 
\begin{equation*}
\label{E4.6.2}\tag{E4.6.2}
c_i=\sum_{j=1}^m a_{ij} e_j,
\end{equation*}
then we have:
\begin{lemma}
\label{lem4.7} $\hat{A}(H)=\{c_1,c_2,\ldots,c_m\}$ and $\Bbbk \hat{A}
(H)=Z(H) $.
\end{lemma}

\begin{proof}
$\hat{A}(H)=\{c_1,c_2,\ldots,c_m\}$ is clear by \eqref{E4.6.1}.

Since for any $\chi_\alpha\in \Gamma$ and $1\leq i,j\leq m$,
\begin{equation*}
\frac{\chi_\alpha}{\chi_\alpha (1)}(c_{i}c_{j})= a_{i\alpha} a_{j\alpha}
=\frac{\chi_\alpha }{\chi_\alpha (1)}(c_{i}) \frac{\chi_\alpha}{\chi_\alpha (1)}(c_{j}).
\end{equation*}
Then the product $c_{i}\cdot c_{j}$ in $\Bbbk \hat{A}(H)$
is exactly $c_{i}c_{j}$ by the definition of $c_{i}\cdot c_{j}$ 
\eqref{E2.7.1}.

Therefore, $\Bbbk \hat{A}(H)=Z(H)$.
\end{proof}

For $1\leq i\leq m$, let $D^{i}=\diag(a_{1i},a_{2i},\ldots ,a_{mi})$, 
$\hat{D}^{i}=\diag(a_{i1},a_{i2},\ldots ,a_{im})$, 
$B_{i }$ be the left multiplication matrix of 
$\frac{\chi _{i }}{\chi_{i }(1)}$ on $C(H)$  
and $\hat{B}_{i}$ be the left multiplication
matrix of $c_{i}$ on $Z(H)$, then, by a general Verlinde formula 
for the fusion rules \cite[Theorem 3.1, Theorem 3.8]{CW2010a}, we have

\begin{proposition}
\label{pro4.8} Let $H$ be a finite dimensional semisimple Hopf algebra with
a commutative character ring over $\Bbbk $ such that $H$ has $m$
non-isomorphic irreducible characters. $B_{i},\hat{B}_{i},D^{i},\hat{D}^{i}$ 
$\left( 1\leq i\leq m\right) $ and $E$ are defined as above, 
then for $1\leq i\leq m$,
\begin{enumerate}
\item[(1)] $B_i=E D^i E^{-1}$,

\item[(2)] $\hat{B}_i=E^{T} \hat{D}^i (E^T)^{-1}$.
\end{enumerate}
\end{proposition}

\begin{remark}
\label{rem4.9} In fact, for any $1\leq j\leq m$, the $j$-th column of $E$ 
is one common eigenvector for
all $B_{i}$ with the corresponding eigenvalue being the $(i, j)$-entry in $E$
for all $1\leq i\leq m$.
\end{remark}

It's well-known that for an Artinian simple algebra $R$,  
if $R=\bigoplus\limits_{i=1}^m I_i$ is a direct sum of some
left ideals, then there exists a set of minimal orthogonal idempotents 
$S^{(i)}\subset I_i$ such that $I_i$ is generated by elements in $S^{(i)}$
as left $R$-module.  Let $S=\{f_1,\cdots, f_n\}$ be the union of all $S^{(i)}$
$(1\leq i\leq m)$. Since for $1\leq u,v\leq n$, 
$f_u R f_v\cong \Hom_R (R f_v, R f_u)$ is 1-dimensional, we can find an element
$r_{uv}$ in $f_u R f_v$ such that $\{r_{uv} \mid 1\leq u,v\leq n\}$ 
is a basis of $R$ satisfying 
$r_{uv}r_{kl}=\delta_{vk}r_{ul}$ for $1\leq u,v,k,l\leq n$, 
and $I_i$ is the space spanned by $\{r_{uv}\mid 1\leq u\leq n, v\in S^{(i)}\}$
for each $1\leq i\leq m$. 

Therefore, as a cosemisimple Hopf algebra, $H$ has 
a coalgebra decomposition
$$H=\bigoplus_{k=1}^z H_k$$
and each $H_k$ is a simple subcoalgebra with dimension $m_k^{2}$. 
For each $1\leq k\leq z$, since 
$H_k=\bigoplus\limits_{k=i}^m H_k\cap \mathcal{C}_i$, 
we can find a basis $\{\varepsilon_{uv}^k\mid 1\leq u,v\leq m_k\}$ for $H_k$
and a subset $I_{i}^{(k)}\subset \{1, \cdots, m_k\}$ for $1\leq i\leq m$,
satisfying 
\begin{equation*}\label{E4.9.1}\tag{E4.9.1}
\Delta(\varepsilon _{uv}^{k}) = \sum_{l=1}^{m_k} \varepsilon_{ul}^{k}\otimes 
\varepsilon_{lv}^{k} \quad \text{and} \quad 
\varepsilon(\varepsilon _{uv}^{k}) = \delta_{uv},
\end{equation*}
and $H_k \cap \mathcal{C}_i$ is the $\Bbbk$-space spanned by
$\{\varepsilon_{uv}^k\mid 1\leq u\leq m_k, v\in I_{i}^{(k)}\}$.

For each $i$, let $F_i=\{(l,k)\mid 1\leq k\leq z, l\in I_i^{(k)}\}$. Then
$\mathcal{C}_{i}$ is the $\Bbbk$-space spanned by 
$\{\varepsilon _{uv}^{k}\mid (u,k)\in F_i \text{ and } 1\leq v \leq m_k\}$.

Let $\{(\varepsilon_{uv}^k)^*\mid 1\leq k\leq m \text{ and } 1\leq u,v\leq m_k\}$
be the dual basis of 
$\{\varepsilon_{uv}^k\mid 1\leq k\leq m \text{ and } 1\leq u,v\leq m_k\}$.
By \cite[Proposition 1]{LR1988b}, we can use the above dual bases 
to express the integral of $H$:
\begin{equation*}
T=\sum_{k=1}^z\sum_{u,v=1}^{m_k} ((\varepsilon_{uv}^k)^*\rightharpoonup
\varepsilon_{uv}^k) =\sum_{k=1}^z m_k(\sum_{u=1}^{m_k}
\varepsilon_{uu}^k)\in \int_H
\end{equation*}
and $\varepsilon(T)=\dim(H)$.

Let $T=\displaystyle\sum_{i=1}^m T_i$, where $T_i\in \mathcal{C}_i$. Then, by
Lemma \ref{lem4.5}, 
\begin{equation*}
T=\sum t_{(2)}TS(t_{(1)})=\sum_{i=1}^m(\sum t_{(2)}T_iS(t_{(1)}))
=\sum_{i=1}^m\varepsilon(T_i)c_i.
\end{equation*}
By the uniqueness of decomposition of the direct sum, 
$T_i=\varepsilon(T_i)c_i$. On the other hand, 
\begin{equation*}
T_i=\sum_{(l,k)\in F_i} m_k\varepsilon_{ll}^k,
\end{equation*}
then $\varepsilon(T_i)=\dim(\mathcal{C}_i)$. Therefore, $T_i=C_i$.

It's obvious to see $1\in \hat{A}(H)$ since $\pi(1)$ is a 
group-like element in $H/K$. Without lost of generality, let $c_1=1$.
Let $\Lambda\in \int_{H^*}$ and $\langle \Lambda,1\rangle =1$, then $\langle
\Lambda,c_i\rangle=\delta_{1i}$. By \cite[Proposition 2]{Zs1993}, we have 
\begin{equation*}
\sum \Lambda_{(1)}\otimes S^*(\Lambda_{(2)})= \sum_{k=1}^{z} \frac{1}{m_k}
\sum_{u,v=1}^{m_k} (\varepsilon_{uv}^k)^*\otimes(\varepsilon_{vu}^k)^*.
\end{equation*}

Therefore,
\begin{eqnarray*}
\langle \Lambda,T_iT_{j^*}\rangle &=& \sum_{k=1}^{z} \frac{1}{m_k}
\sum_{u,v=1}^{m_k} \langle
(\varepsilon_{uv}^k)^*\otimes(\varepsilon_{vu}^k)^*, T_i\otimes T_j\rangle \\
&=& \sum_{k=1}^{z} \frac{1}{m_k}\sum_{u,v=1}^{m_k}
\langle(\varepsilon_{uv}^k)^*,\sum_{(l,y)\in F_i}
m_y\varepsilon_{ll}^y\rangle \langle(\varepsilon_{vu}^k)^*,\sum_{(l^{\prime
},y^{\prime })\in F_j} m_{y^{\prime }}\varepsilon_{l^{\prime }l^{\prime
}}^{y^{\prime }}\rangle \\
&=& \delta_{ij}\sum_{(l,k)\in F_i} m_k=\delta_{ij} \dim(\mathcal{C}_i).
\end{eqnarray*}

Then we obtain 
\begin{equation*}
\hat{p}_1(i, j)=\delta_{ij^*}\frac{1}{\dim(\mathcal{C}_i)}.
\end{equation*}

Let $\displaystyle\hat{s}(i)=\frac{1}{\hat{p}_1(i, i^*)}$, then

\begin{proposition}
\label{pro4.10} For $ 1\leq i\leq m$, $\hat{s}(i)=\dim(\mathcal{C}_i)$.
\end{proposition}

By Proposition \ref{pro2.8}, we have the following orthogonality relations
over $H$.

\begin{theorem}[{cf. \cite[Theorem 2.5]{CW2014}}]
\label{thm4.11} Let $H$ be a semisimple Hopf algebra over an algebraically
closed field $\Bbbk $ with $\ch(\Bbbk )=0$, $\{\chi _{1},\cdots ,\chi _{m}\}$
be all the irreducible characters of $H$, 
and $\{\mathcal{C}_{1},\cdots ,\mathcal{C}_{m}\}$ be all the conjugacy classes.
Let $E=(a_{ij})$ be defined   
in \eqref{E4.6.2}, then 
\begin{equation*}
\label{E4.11.1}\tag{E4.11.1}
\sum_{i=1}^{m}\chi _{\alpha }^{2}(1)\dim (\mathcal{C}_{i})a_{i\alpha} 
a_{i^{\ast }\beta }=\dim (H)\delta _{\alpha \beta },
\quad 1\leq \alpha,\beta\leq m,
\end{equation*}
and 
\begin{equation*}
\label{E4.11.2}\tag{E4.11.2}
\sum_{\alpha =1}^{m}\chi _{\alpha }^{2}(1)\dim (\mathcal{C}_{i})a_{i\alpha} a_{j^{\ast }\alpha }=\dim (H)\delta _{ij}, \quad 1\leq i,j\leq m.
\end{equation*}
\end{theorem}

\begin{remark}
\label{rem4.12} 
\eqref{E4.11.1} can also be written as 
\begin{equation*}  \label{E4.12.1}
\sum_{i=1}^m \dim(\mathcal{C}_i)\chi_\alpha(c_i)\chi_\beta(S(c_i))
=\dim(H)\delta_{\alpha\beta}.  \tag{E4.12.1}
\end{equation*}
When $H=\Bbbk G$ for a finite group $G$, 
since $\chi(g)=\chi(c_i)$ if $g\in C_i$,
the left side of \eqref{E4.12.1} is 
\begin{equation*}
\sum_{i=1}^m \sum_{g\in C_i}\chi_\alpha(g)\chi_\beta(g^{-1})= \sum_{g\in G}
\chi_\alpha(g)\chi_\beta(g^{-1}),
\end{equation*}
then \eqref{E4.12.1} is exactly the first orthogonality relation for groups.

Similarly, \eqref{E4.11.2} is the second orthogonality relation for groups.
\end{remark}

In \eqref{E4.11.2}, if let $j=i$, then

\begin{corollary}[\protect\cite{Ka1972, Zy1994}]
\label{cor4.13} $\dim (\mathcal{C}_{i})\mid \dim (H)$.
\end{corollary}

By Proposition \ref{pro4.8} and Theorem \ref{thm4.11}, we have:

\begin{corollary}
\label{cor4.14} Let $\displaystyle\chi_i\chi_j=\sum_{k=1}^m N_{ij}^k\chi_k$,
$\displaystyle C_iC_j=\sum_{k=1}^m \hat N_{ij}^k C_k$ and  
write $E^{-1}=(n_{ij})$. Then 
$$n_{ij}= \frac{\chi^{2}_j(1)a_{i^*j}\dim(\mathcal{C}_i)}{\dim(H)},$$
and for $ 1\leq i,j,k\leq m$,
\begin{enumerate}
\item[(1)] $B_i=(b^i_{uv})$, where $\displaystyle b^i_{uv}
=\sum_{l=1}^m \frac{a_{li} a_{lu}
a_{l^*v}\chi_v^{2}(1)\dim(\mathcal{C}_l)} {\dim(H)}$,

\item[(2)] $\displaystyle N_{ij}^k=\frac{\chi_i(1)\chi_j(1)\chi_k(1)}{\dim(H)}
 \sum_{l=1}^m a_{li} a_{lj} a_{l^*k}\dim(\mathcal{C}_l)$,

\item[(3)] $\hat B_i=(\hat{b}^i_{uv})$, where $\hat{b}^i_{uv}=
\displaystyle \sum_{l=1}^m \frac{a_{ul} a_{il}
a_{v^*l}\chi_l^{2}(1)\dim(\mathcal{C}_v)} {\dim(H)}$, and

\item[(4)] $\displaystyle \hat N_{ij}^k=\frac{\dim(\mathcal{C}_i)
\dim(\mathcal{C}_j)} {\dim(H)} \sum_{l=1}^m a_{il} a_{jl} a_{k^*l}\chi_l^{2} (1)$.
\end{enumerate}
\end{corollary}

\section{Special case: the Drinfeld double}

In this section, we will study a special case -- the Drinfeld double, and
then deduce some properties for semisimple quasitriangular Hopf algebras.

If $H$ is a finite dimensional semisimple Hopf algebra over $\Bbbk$, 
so is the Drinfeld double $D(H)$ \cite[Proposition 7]{Ra1993}.

\begin{lemma}
\label{lem5.1} 
Let all the isomorphism classes of simple representations
of $D(H)$ be $\{V_1,\cdots, V_m\}$, 
then the change of basis matrix $E=(a_{ij})$ 
defined in \eqref{E4.6.2} is symmetric.
\end{lemma}

\begin{proof}
Recall that $\Repr(D(H))$ is a modular category \cite{BK2000}, and there exists
a matrix $\tilde{s}=(\tilde{s}_{ij})$ with 
$\tilde{s}_{ij}=(tr|_{V_i}\otimes tr|_{V_{j^*}})(R^{21}_{D(H)}R_{D(H)})$.
Hence, for $1\leq i\leq m$, $\tilde{s}_{1i}=\tilde{s}_{i1}=\dim V_i$.

Let $D=\dim H$ and $s_{ij}=\frac{\tilde{s}_{ij}}{D}$. 
By Verlinde formula (\cite[Theorem 3.1.13]{BK2000}), 
\begin{equation*}
N_{ij}^k=\displaystyle\sum_{r=1}^m\frac{s_{ir}s_{jr}s_{k^*r}}{s_{1r}}.
\end{equation*}
Hence, by the properties of the matrix $s=(s_{ij})$, for $ 1\leq t\leq m$, 
\begin{eqnarray*}
\sum_{k=1}^m N_{ij}^k \frac{s_{kt}}{s_{1t}} &=& \sum_{k=1}^m\sum_{r=1}^m 
\frac{s_{ir}s_{jr}s_{k^*r}s_{tk}}{s_{1r}} \\
&=& \sum_{r=1}^m\frac{s_{ir}s_{jr}}{s_{1r}s_{1t}}(\sum_{k=1}^m
s_{tk}s_{kr^*}) \\
&=& \sum_{r=1}^m\frac{s_{ir}s_{jr}}{s_{1r}s_{1t}}\delta_{tr} \\
&=& \frac{s_{it}}{s_{1t}}\frac{s_{jt}}{s_{1t}}.
\end{eqnarray*}

Since 
\begin{equation*}
p_k(i,j)=\frac{N_{ij}^k\dim(V _k)}{\dim(V_i)\dim(V_j)}= \frac{N_{ij}^k
\tilde{s}_{1k}}{\tilde{s}_{1i}\tilde{s}_{1j}},
\end{equation*}
we have $a_{ij}=\displaystyle\frac{s_{ij}}{Ds_{1i}s_{1j}}$ and 
$E$ is symmetric.
\end{proof}

\begin{corollary}
\label{cor5.2} If $H$ is a finite dimensional semisimple Hopf algebra, then
\begin{equation*}
\hat A(D(H))\simeq A(D(H)).
\end{equation*}
\end{corollary}

\begin{corollary}
\label{cor5.3} Let $C$ be a conjugacy class of $D(H)$, then there is a character
$\chi\in Irr(D(H))$, such that $\dim(C)=\chi^{2}(1)$.
\end{corollary}

\begin{proof}
For $C$, there exists $a\in \hat A(D(H))$ such that $\hat{s}(a)=\dim(C)$.
Since $\phi: \hat A(D(H))\simeq A(D(H))$, then 
$s(\phi(a))=\hat{s}(a)=\dim(C) $.

On the other hand, there exists $\chi\in Irr(D(H))$ such that 
$s(\phi(a))=\chi^{2}(1)$. Hence the result holds.
\end{proof}

\begin{proposition}
\label{pro5.4} The orthogonality relation over a semisimple Drinfeld double 
$D(H)$ is 
\begin{equation*}  \label{E5.4.1}\tag{E5.4.1}
\sum_{\alpha=1}^m \chi_\alpha^{2}(1)\chi_i^{2}(1) a_{i\alpha} a_{j^*\alpha}
=\dim(D(H))\delta_{ij}, \quad 1\leq i,j\leq m.
\end{equation*}
\end{proposition}

\begin{proof}
Since $a_{i\alpha}=\displaystyle\frac{s_{\alpha i}}{Ds_{1\alpha}s_{1i}}= 
\frac{Ds_{\alpha i}}{\chi_\alpha(1)\chi_i(1)}$ and $a_{i\alpha}=a_{\alpha i}$, 
\begin{eqnarray*}
\sum_{\alpha=1}^m \chi_\alpha^{2}(1)\chi_i^{2}(1)a_{i\alpha} a_{j^*\alpha} &=& 
\frac{D^{2}\chi_i(1)}{\chi_j(1)}(\sum_{\alpha=1}^m s_{i\alpha}s_{\alpha j^*})
\\
&=& \frac{\dim(D(H))\chi_i(1)}{\chi_j(1)}\delta_{ij} \\
&=& \dim(D(H))\delta_{ij}  
\end{eqnarray*}
\end{proof}

Furthermore, if $H$ is quasitriangular with $R$, 
then any $H$-module $M$ is also a natural $D(H)$-module 
induced by the surjection of Hopf algebras 
\begin{equation*}
\Phi :D(H)\longrightarrow H,\quad p\bowtie h\mapsto \langle
p,R^{1}\rangle R^{2}h.
\end{equation*}
In this case, let $j=i$ in \eqref{E5.4.1}, then we have

\begin{corollary}[{{\protect\cite[Theorem 1.5]{EG1998}} }]
\label{cor5.5} For any finite dimensional semisimple quasitriangular Hopf
algebra $H$ over an algebraically closed field $\Bbbk $ of characteristic
zero, if $\chi $ is an irreducible character of $H$, then $\chi (1)\mid \dim (H)
$.
\end{corollary}

M. Cohen and S. Westreich \cite{CW2011} proved that 
the product of two class sums 
of $H$ is an integral combination up to a factor of $\dim(H)^{-2}$ of 
the class sums of $H$. Here, by using probability groups, 
we have the following theorem:

\begin{theorem}
\label{thm5.6} Let $H$ be a finite dimensional semisimple quasitriangular
Hopf algebra over an algebraically closed field $\Bbbk $ of characteristic
zero and $\dim (H)=d$, then the product of two class sums is an integral
combination up to a factor of $d^{-1}$ of the class sums of $H$, i.e.
$$C_iC_j=\frac{1}{d}\sum_{k=1}^m \hat{N}_{ij}^k C_k,$$
where $\{C_i\mid1\leq i\leq m\}$ is the set of class sums of $H$ and  
$\hat{N}_{ij}^k\in \N_{\geq 0}$. 
\end{theorem}

Before the proof, we introduce some notations and lemmas first:

\begin{enumerate}
\item[$\bullet $] $Irr(H)=\{\chi _{1},\ldots ,\chi _{m}\},$

\item[$\bullet $] $Irr(D(H))=\{\hat{\chi}_{1},\ldots ,\hat{\chi}_{r}\},$

\item[$\bullet $] the class sums of $D(H)$ are $\{\hat{C}_1,\cdots,\hat{C_{r}}\}$,

\item[$\bullet$] $\displaystyle\hat\chi_j \hat\chi_j=\sum_{k=1}^r
N_{ij}^k\hat\chi_k$, for $1\leq i,j\leq r$.
\end{enumerate}

Then we have

\begin{enumerate}
\item[$\bullet$] $\displaystyle A(H)=\{\frac{\chi_i}{\chi_i(1)}\}_{i=1}^m$
is a probability group with the probability map $p$,

\item[$\bullet$] $\displaystyle \hat A(H)=\{c_i=\frac{C_i}{\varepsilon_H(C_i)}\}_{i=1}^m$ 
is a probability group with the probability map  $\hat p$,

\item[$\bullet$] $\displaystyle A(D(H))=\{\frac{\hat\chi_i}{\hat\chi_i(1)}\}_{i=1}^r$ 
is a probability group with the probability map  $P$, and

\item[$\bullet$] $\displaystyle \hat A(D(H))=\{\hat{c}_i=\frac{\hat C_i} 
{\varepsilon_{D(H)}(\hat C_i)}\}_{i=1}^r$ is a probability group with 
the probability map $\hat P$.
\end{enumerate}

Using the isomorphism of vector spaces
\begin{equation*}
F:(D(H))^{\ast }\longrightarrow D(H),\quad p\mapsto 
\sum\langle p,R_{D(H)}^{1}r_{D(H)}^{2}\rangle R_{D(H)}^{2}r_{D(H)}^{1},
\end{equation*}
without loss of generality, we assume that $\hat\chi_i=\Phi^*(\chi_i)$ for 
$1\leq i\leq m$ and $\hat C_j=\hat\chi_j(1)F(\hat \chi_j)$ for 
$1 \leq j\leq r$ by \cite[(15)]{CW2010b}.

Let $A_i$ be the set of index $j$ such that $\hat{c}_j|_{A(H)}=c_i$. Since 
$\mathcal{C}_i$ is a simple $D(H)$-module, the corresponding character is 
in $Irr(D(H))$. Let it be $\hat\chi_{\beta(i)}$.

\begin{lemma}
\label{lem5.6} $\beta(i)\in A_i$.
\end{lemma}

\begin{proof}
Let $\{h_k\mid 1\leq k\leq d\}$ and $\{h_k^*\mid 1\leq k\leq d\}$ be dual bases for $H$ and $H^*$ respectively. 
Then for $1\leq j\leq m$, we have: 
\begin{eqnarray*}
\langle \hat \chi_j,\hat C_{\beta(i)}\rangle &=& \langle \Phi^*(\chi_j),
\dim(\mathcal{C}_i)F(\hat\chi_{\beta(i)})\rangle \\
&=& \dim(\mathcal{C}_i)\sum_{k,l=1}^d \langle \Phi^*(\chi_j), \langle
\hat\chi_{\beta(i)}, h_k^*\bowtie h_l\rangle (\varepsilon_H\bowtie
h_k)(h_l^*\bowtie 1_H)\rangle \\
&=& \dim(\mathcal{C}_i) \sum_{k,l=1}^d \langle \hat\chi_{\beta(i)},
h_k^*\bowtie h_l \rangle \langle \chi_j, \Phi(h_l^*\bowtie h_k)\rangle \\
&=& \dim(\mathcal{C}_i) \sum_{k,l=1}^d \langle \hat\chi_{\beta(i)},
h_k^*\bowtie h_l\rangle \langle \chi_j,\langle h_l^*,R^1\rangle R^{2}
h_k\rangle \\
&=& \dim(\mathcal{C}_i) \langle \chi_j, \sum_{k=1}^d \langle
\hat\chi_{\beta(i)}, h_k^*\bowtie R^1\rangle R^{2} h_k\rangle .
\end{eqnarray*}
Let $K_i=\displaystyle\sum_{k=1}^d \langle \hat\chi_{\beta(i)},
h_k^*\bowtie R^1\rangle R^{2} h_k$.

Since $\hat\chi_{\beta(i)}$ is the character of $\mathcal{C}_i$, 
using the basis $\{\varepsilon_{uv}^k\}$ for $H$, the dual basis 
$\{(\varepsilon_{uv}^k)^*\}$ for $H^*$ and the index set $F_i$ 
in \eqref{E4.9.1}. we have 
\begin{eqnarray*}
K_i &=& \sum_{k=1}^z \sum_{u',v'=1}^{m_k} \langle \hat\chi_{\beta(i)}, (\varepsilon_{u'v'}^k)^*\bowtie R^1 \rangle R^{2} \varepsilon_{u'v'}^k  \\
&=& \sum_{k=1}^z \sum_{u',v'=1}^{m_k} (\sum_{(l,x)\in F_i}\sum_{u=1}^{m_x} 
 \langle(\varepsilon_{lu}^x)^* *(\varepsilon_{u'v'}^k)^*, R^1\cdot_{ad}
  \varepsilon_{lu}^x \rangle R^{2} \varepsilon_{u'v'}^k)  \\
&=&\sum_{(l,x)\in F_i}\sum_{u',v'=1}^{m_x} 
 \langle(\varepsilon_{lv'}^x)^*, R^1\cdot_{ad}
  \varepsilon_{lu'}^x \rangle R^{2} \varepsilon_{u'v'}^x.  
 \end{eqnarray*}
Hence, by \cite[Proposition 3.1]{LZ2019}, $K_i\in \mathcal{C}_i$. As a
result, by Lemma \ref{lem4.5} , 
\begin{eqnarray*}
\langle \chi_j, K_i\rangle &=& \langle \chi_j, t_{(1)}K_iS(t_{(2)})\rangle \\
&=& \langle \chi_j, \varepsilon(K_i)c_i\rangle \\
&=& \langle \chi_j, \dim(\mathcal{C}_i)c_i\rangle =\langle \chi_j,
C_i\rangle .
\end{eqnarray*}

Therefore, for $1\leq j\leq m$, 
\begin{equation*}
\langle \hat \chi_j,\hat C_{\beta(i)}\rangle =\dim(\mathcal{C}_i)\langle
\chi_j,C_i\rangle \text{ and } \langle \hat \chi_j,\hat{c}_{\beta(i)}
\rangle =\langle \chi_j,c_i\rangle.
\end{equation*}
This implies $\beta(i)\in A_i$.
\end{proof}

Now we begin to prove Theorem \ref{thm5.6}:

\begin{proof}
By \cite[Proposition 2.11]{Ha1979}, 
$\hat{A}(H)\cong \hat{A}(D(H))//(A(H))^{\perp }$,
where $(A(H))^{\perp }=\{f\in A(D(H))\mid f(a)=1,\forall a\in A(H)\}$,
and in $\hat{A}(D(H))//(A(H))^{\perp }$, $[\hat{c}_u]=[\hat{c}_v]$ if
and only if there is an integer $i\in [1, m]$ such that $u,v\in A_i$.

By Proposition \ref{pro2.5} and Corollary \ref{cor5.2}, for all 
$s\in A_{i},u\in A_{j}$, we have: 
\begin{equation*}
\hat{p}_{k}(i,j)=\sum_{v\in A_{k}}\hat{P}_{v}(s,u)=\sum_{v\in
A_{k}}P_{v}(s,u)=\sum_{v\in A_{k}}
\frac{N_{su}^{v}\hat{\chi}_{v}(1)}{\hat{\chi}_{s}(1)\hat{\chi}_{u}(1)}.
\end{equation*}
Since $C_{i}=\dim (\mathcal{C}_{i})c_{i}$ ,
\begin{equation*}
C_{i}C_{j}=\sum_{k=1}^{m}( \sum_{v\in A_{k}}\frac{N_{su}^{v}\hat{\chi}_{v}(1)
\dim (\mathcal{C}_{i})\dim (\mathcal{C}_{j})}
{\hat{\chi}_{s}(1)\hat{\chi}_{u}(1)\dim (\mathcal{C}_{k})} ) C_{k}.
\end{equation*}

By Lemma \ref{lem5.6}, $\beta(i)\in A_i$. Let $s=\beta(i),u=\beta(j)$, then
\begin{equation*} 
C_i C_j=\sum_{k=1}^m (\sum_{v\in A_k} \frac{N_{\beta(i)\beta(j)}^v
\hat{\chi}_v(1)} {\dim(\mathcal{C}_k)})C_k, 
\end{equation*}
since $\hat\chi_{\beta(i)}(1)=\dim(\mathcal{C}_i),\hat\chi_{\beta(j)}(1)=
\dim(\mathcal{C}_j)$.

By Corollary \ref{cor4.13}, $\dim(\mathcal{C}_k) | \dim(H) $. Then 
$$\sum_{v\in A_k} \frac{\dim(H) N_{\beta(i)\beta(j)}^v
\hat{\chi}_v(1)} {\dim(\mathcal{C}_k)} \in \Z_{\geq 0},$$ 
which is $\hat{N}_{ij}^k$.
\end{proof}

\begin{remark}
\label{rem5.7} In the proof, if we take $j=i^{\ast }$ and $u=s^{\ast }$,
then $\dim (\mathcal{C}_{i})\mid \hat{\chi}_{s}^{2}(1)$ for all $s\in A_{i}$, 
since $\displaystyle\frac{1}{\dim (\mathcal{C}_{i})}
=\hat{p}_{1}(i,i^{\ast})=
\sum_{v\in A_{1}}\frac{N_{ss^{\ast }}^{v}\chi_{v}(1)}{\chi_{s}^{2}(1)}$.
\end{remark}

\section{2-integral probability groups with 2 or 3 elements }

Recall that an $r$-integral probability group $A$
(see Remark \ref{rem3.4} or \cite{Ha1979}) satisfies 
\begin{enumerate}
\item[(*)] for all $ a\in A$, $s(a)^{\frac{1}{r}}$ is an integer, and

\item[(**)] for all $ a,b,c\in A$, $p(a\cdot b=c)s(a)^{\frac{1}{r}}s(b)^
{\frac{1}{r}}/s(c)^{\frac{1}{r}}$ is a non-negative integer.
\end{enumerate}

It is obvious that in an $r$-integral probability group $A$, $p(a\cdot
b=c)\in \mathbb{Q}$ for any $a,b,c\in A$,.

\begin{lemma}
\label{lem6.1} Let $A$ be a probability group, then for all $ a,b,c\in A$, 
$$p(a\cdot b=c^{-1})s(a)=p(b\cdot c=a^{-1})s(c).$$
\end{lemma}

\begin{proof}
In \eqref{E2.1.2}, we let $d=1$, the unit element in $A$. Then 
\begin{eqnarray*}
p(a\cdot b=c^{-1})p(c^{-1}\cdot c=1) &=& \sum_{x \in A} p(a\cdot b=x)p(x\cdot c=1) \\
&=& \sum_{y \in A} p(a\cdot y=1)p(b\cdot c=y) \\
&=& p(a\cdot a^{-1}=1)p(b\cdot c=a^{-1}).
\end{eqnarray*}
Since $s(a)=\displaystyle\frac{1}{p(a\cdot a^{-1}=1)}$, the lemma holds.
\end{proof}

\begin{lemma}
\label{lem6.2} For all $ a, b\in A$, if $a=a^{-1}$, then $p(a\cdot
b=a)=p(b\cdot a=a)$.
\end{lemma}

\begin{proof}
Let $c=a$ in Lemma \ref{lem6.1}.
\end{proof}

When $\left\vert A\right\vert =2$, assume that $A=\{1,a\}$ with $s(a)=n^{2}$ 
for some positive integer $n$. 
By Condition (**), 
$p(a\cdot a^{-1}=a)s(a)^{\frac{1}{2}}\in \Z$, i.e.
$$p(a\cdot a^{-1}=a)s(a)^{\frac{1}{2}}
=(1-\frac{1}{n^{2}}) n= \frac{n^{2}-1}{n}\in \Z.$$
Hence $n=1$ and $A=A(\Bbbk \mathbb{Z}_{2})$.

When $\left\vert A\right\vert =3$, assume that $A=\{1,a,b\}$. There are two
cases: (1) $a^{-1}=a, b^{-1}=b$, and (2) $a^{-1}=b$.

\textbf{CASE I}: $a^{-1}=a, b^{-1}=b$.

In this case, $p(a\cdot b=1)=p(b\cdot a=1)=0$ by the uniqueness of inverse
element.

Let $s(a)=n_1^{2},s(b)=n_2^{2}$ for some positive integers $n_1,n_2$. By Lemma 
\ref{lem6.2}, the probability group is abelian.

Let $p(a\cdot a=a)=\displaystyle\frac{m_1}{n_1}$ and $p(b\cdot b=b)=
\displaystyle\frac{m_2}{n_2}$, where $m_1,m_2\in \Z^{\geq 0}$ 
and $m_i\leq n_i(i=1,2)$. By Condition (**), 
\begin{equation*}
p(a\cdot a=b)s(a)/s(b)^{\frac{1}{2}} =\frac{n_1^{2}-m_1n_1-1}{n_2}\in \Z.
\end{equation*}
Thus there exists an integer $k\in \mathbb{Z}$ 
satisfying $(n_1-m_1)n_1+kn_2=1$. As a result, 
\begin{equation*}  \label{E6.3.1}
\gcd (n_1,n_2)=1.  \tag{E6.3.1}
\end{equation*}

Let $p(a\cdot b=a)=\displaystyle\frac{u}{v}$, where $u\in \Z^{> 0}$, 
$v\in \Z^{+}$ and $\gcd (u,v)=1$. Then $p(a\cdot b=b)=\displaystyle 1-\frac{u}{v}
=\frac{v-u}{v}$.  

Since $p(a\cdot b=a)s(b)^{\frac{1}{2}}\in \Z$, 
$\displaystyle\frac{un_{2}}{v}\in \Z$. One can also get 
$\displaystyle\frac{(v-u)n_{1}}{v}\in \Z$. Hence $v | \gcd(n_{1},n_{2})$.
By \eqref{E6.3.1}, $v=1$, then $(u,v-u)=(1,0)$ or $(0,1)$.

Without loss of generality, let $u=1$. Hence $p(a\cdot b=a)=1,
p(a\cdot b=b)=0 $. By Lemma \ref{lem6.1}, $p(b\cdot b=a)=0$, then 
\begin{eqnarray*}
1 &=& p(b\cdot b=1)+p(b\cdot b=b) \\
&=& \frac{1}{n_2^{2}}+\frac{m_2}{n_2} \\
&=& \frac{1+m_2n_2}{n_2^{2}}.
\end{eqnarray*}
This forces that $n_2=1$ and $m_2=0$. At the same time, 
one can get $n_1=2$ and $m_1=1$. 
Thus $A=A(\Bbbk \mathcal{S}_3)$.

\textbf{CASE II}: $a^{-1}=b$.

In this case, $p(a\cdot a=1)=p(b\cdot b=1)=0$ by the uniqueness of inverse
element.

Assume $s(a)=s(b)=n^{2}$ for some positive integer $n$. 
By Lemma \ref{lem6.1}, 
$p(a\cdot b=b)=p(b\cdot a=b)$. Hence $A$ is abelian. 
Also by \eqref{E2.1.2}, 
\begin{equation*}
\sum_{x \in A} p(a\cdot a=x)p(x\cdot b=1)=\sum_{x\in A} p(a\cdot x=1)p(a\cdot b=x).
\end{equation*}
Since $a^{-1}=b$, then 
\begin{equation*}  \label{E6.3.2}
p(a\cdot a=a)=p(a\cdot b=b).  \tag{E6.3.2}
\end{equation*}
Meanwhile, one can also get 
\begin{equation*}  \label{E6.3.3}
p(a\cdot b=a)=p(b\cdot b=b).  \tag{E6.3.3}
\end{equation*}
By \eqref{E2.1.5} and $a^{-1}=b$, 
\begin{equation*}  \label{E6.3.4}
p(a\cdot a=a)=p(b\cdot b=b).  \tag{E6.3.4}
\end{equation*}
Therefore, by \eqref{E6.3.2}-\eqref{E6.3.4}, 
\begin{equation*}
p(a\cdot b=b)=p(a\cdot b=a).
\end{equation*}

By Condition (**), $p(a\cdot b=a)n\in \mathbb{Z}$. Let 
$p(a\cdot b=a)=\displaystyle\frac{m}{n}$, where $m\in \Z^{\geq 0}$ and 
$m\leq n $. By \eqref{E2.1.1}, 
\begin{equation*}
\frac{1}{n^{2}}+\frac{2m}{n}=1.
\end{equation*}
That is $1=(n-2m)n$. Thus $n=1$, and $A=A(\Bbbk\mathbb{Z}_3)$.

\begin{proposition}
There are only two types of 2-integral probability groups with 3 elements, 
$A(\Bbbk \mathcal{S}_3)$ and $A(\Bbbk \mathbb{Z}_3)$.
\end{proposition}

\subsection*{Acknowledge}
This work was supported by NNSF of China (No. 11331006).

\end{document}